\documentclass[12pt]{article}
\usepackage[utf8]{inputenc}
\usepackage{amsfonts}
\usepackage{scrextend}
\usepackage{mathtools}
\usepackage{float}
\usepackage{amsmath}
\usepackage[margin=1.2in]{geometry}
\usepackage{authblk}
\usepackage{amssymb}
\usepackage{epigraph}
\usepackage[english]{babel}
\usepackage[nottoc]{tocbibind}
\usepackage{tikz-cd}
\usepackage{graphicx}
\usetikzlibrary{matrix}
\usepackage{amsthm, enumitem, cleveref}
\usepackage[mathscr]{euscript}
 \let\mathscr\relax
\usepackage[scr]{rsfso}
\usepackage{comment}

\usepackage{bm,bbm, mathdots, accents}

\theoremstyle{definition}
\newtheorem{nummer}{ }[section]

\newtheorem{thm}[nummer]{\sc Theorem}

\newtheorem{prp}[nummer]{\sc Proposition}
\newtheorem{lem}[nummer]{\sc Lemma}
\newtheorem{cor}[nummer]{\sc Corollary}
\newtheorem{claim}[nummer]{\sc Claim}
\newtheorem{fct}[nummer]{\sc Fact}

\newtheorem{defi}[nummer]{\sc Definition}

\newtheorem*{rmk*}{\sc Remark}
\newtheorem*{obs}{\sc Observation}

\newtheorem{problem}{\sc Problem}

\newcounter{faelle} 
\renewcommand{\thefaelle}{\rm(\alph{faelle})}

\renewcommand\qed{\relax\ifmmode~\hfill$\dashv$\else\unskip\nobreak~\hfill$\dashv$\fi}

\def\epsilon{\varepsilon}

\newcommand\ran{\operatorname{ran}}
\newcommand\dom{\operatorname{dom}}

\newcommand\seq{\operatorname{seq}}

\renewcommand{\phi}{\varphi}
\renewcommand{\theta}{\vartheta}

\newcommand\level{\text{level}}

\newcommand\cls{\text{cl}}

\begin{document}

\begin{center}
{\Large\sc Laver Ultrafilters}\\[1.8ex]


{\small Silvan Horvath}\\[1.2ex] 
{\scriptsize Department of Mathematics, ETH Z\"urich, 
8092 Z\"urich, Switzerland\\ 
silvan.horvath@math.ethz.ch}\\[1.8ex]

{\small Tan {\"O}zalp}\\[1.2ex] 
{\scriptsize Department of Mathematics, University of Notre Dame,
Notre Dame, IN 46556, USA\\
aozalp@nd.edu}\\[1.8ex]

\end{center}


\begin{quote}
{\small {\bf Abstract.} We introduce \emph{Laver ultrafilters}, namely ultrafilters $\mathcal{U}$ for which the associated Laver forcing $\mathbb{L}_{\mathcal{U}}$ has the Laver property. We give simple combinatorial characterisations of these ultrafilters, which allow us to analyse their position among several well-studied combinatorial classes, including $P$-points, rapid ultrafilters, and ultrafilters arising in Baumgartner’s $\mathcal{I}$-ultrafilter framework. In particular, we show that the class of Laver ultrafilters properly contains the class of rapid $P$-points and that it is properly contained both in the class of hereditarily rapid- and in the class of measure zero ultrafilters. Finally, we investigate the (generic) existence of Laver ultrafilters and establish bounds on their generic existence number. In particular, we show that it is consistent that $P$-points do not exist while Laver ultrafilters exist generically.

}
\end{quote}

\begin{quote}
\end{quote}


\setcounter{section}{0}
\section{Introduction and Preliminaries}

Laver forcing originates in Richard Laver’s seminal work on the consistency of the Borel conjecture and the development of iterated forcing techniques \cite{laver}. Since then, it has become a fundamental tool in the study of the set theory of the real line.

One of the key features of Laver forcing is the \emph{Laver property}. This property is important because it ensures a certain degree of control over the reals added by the forcing -- most notably, it prevents the addition of Cohen- and random reals. Moreover, the Laver property is preserved under countable support iterations of proper forcing notions.

\begin{defi}
A forcing notion $\mathbb{P}$ has the \textit{Laver property} if the following holds:

Assume that $\undertilde{f}$ is a $\mathbb{P}$-name for an element of $\omega^\omega$ such that there exists $g \in \omega^\omega$ with $\Vdash_{\mathbb{P}}\undertilde{f}<g$. Then, $\mathbb{P}$ forces that there exists some $S \in \prod_{n \in \omega} [g(n)]^{\leq n+1} $ in the ground model with $\undertilde{f}(n)\in S(n)$ for each $n\in \omega$. \footnote{The bound $n+1$ here is somewhat arbitrary, i.e., it can be replaced by any unbounded $h: \omega \to \omega \setminus \{0\}$. The necessary argument is similar to the proof of Fact \ref{fct:bound}.}
\end{defi}

An important variant of Laver forcing $\mathbb{L}$ is obtained by considering its relativization to an ultrafilter $\mathcal{U}$ on $\omega$, denoted $\mathbb{L}_{\mathcal{U}}$.

\begin{defi}
$\mathbb{L}_{\mathcal{U}}$ is the forcing notion consisting of trees $T \subseteq \omega^{<\omega}$ such that for each $s \in T$ with $s \supseteq \text{stem}(T)$,
\[\text{succ}_T(s):=\{n \in \omega: s^{\smallfrown}n \in T\}\in \mathcal{U},\] ordered by reverse inclusion.
\end{defi}

At this stage, it is natural to ask for which ultrafilters $\mathcal{U}$ the forcing notion $\mathbb{L}_{\mathcal{U}}$ has the Laver property. The main objective of this paper is to study these ultrafilters, which we will call \emph{Laver ultrafilters}.

A similar question arises for the closely related (relativized) Mathias forcing $\mathbb{M}_{\mathcal{U}}$, since its non-relativized variant also has the Laver property, and since $\mathbb{M}_{\mathcal{U}}$ and $\mathbb{L}_{\mathcal{U}}$ are forcing equivalent if $\mathcal{U}$ is a Ramsey ultrafilter\footnote{An ultrafilter $\mathcal{U}$ on $\omega$ is called \emph{Ramsey} or \emph{selective} if it contains witnesses to Ramsey's theorem for pairs.}. In this setting, it is well-known that $\mathbb{M}_{\mathcal{U}}$ retains the Laver property if and only if $\mathcal{U}$ is in fact a Ramsey ultrafilter. In contrast, much less is known about the corresponding question for $\mathbb{L}_{\mathcal{U}}$. Previous work includes the result of B{\l}aszczyk and Shelah \cite{blaszczyk2001regular} showing that $\mathbb{L}_{\mathcal{U}}$ does not add Cohen reals if and only if $\mathcal{U}$ is a \emph{nowhere dense ultrafilter} (defined below). Very recently, Nieto-de la Rosa, Guzm\'an and Ramos-Garc\'ia \cite{nietoguzmanramos26} independently explored a closely related research direction: Among various other results, they characterise the ideals $\mathcal{I}$ on $\omega$ for which the forcing notion $\mathbb{L}_{\mathcal{I}^{+}}$ has the Laver property, in terms of the Kat\v{e}tov order.   One of their results states that it is consistent that there exists a $P$-point $\mathcal{U}$ such that $\mathbb{L}_{\mathcal{U}}$ neither adds Cohen reals nor has the Laver property. As a consequence of our Propositions \ref{prp:rapid} and \ref{prp:semiselective}, this phenomenon occurs precisely in case the $P$-point $\mathcal{U}$ is not simultaneously a rapid ultrafilter.

Let us now recall some of the well-studied combinatorial classes of ultrafilters.

\begin{defi}
Let $\mathcal{U}$ be an ultrafilter over $\omega$.
\begin{enumerate}[label=(\roman*)]
\item $\mathcal{U}$ is called a $P$-point if for every sequence $\langle x_n: n \in \omega\rangle$ of elements of $\mathcal{U}$, there exists $x\in \mathcal{U}$ with $\forall n \in \omega: x \subseteq^{\ast}x_n$.
\item $\mathcal{U}$ is called \emph{rapid} if for every strictly increasing $f\in \omega^\omega$, there exists $x\in \mathcal{U}$ with $\forall n \in \omega: |x \cap f(n)|\leq n$.
\item Let $\mathcal{I}$ be a collection of subsets of some set $X$, such that $\mathcal{I}$ is closed under subsets and contains all singletons. Baumgartner \cite{baumgartner-uf} introduced the following concept: an ultrafilter $\mathcal{U}$ over $\omega$ is called an \textit{$\mathcal{I}$-ultrafilter} if for every $F: \omega \to X$, there exists some $x\in \mathcal{U}$ such that $F[x]\in \mathcal{I}$.
\end{enumerate}
\end{defi}

Baumgartner investigated $\mathcal{I}$-ultrafilters for various such collections $\mathcal{I}$ on $X=2^\omega$. Examples include the discrete sets, the scattered sets, the sets with closure of measure zero and the nowhere dense sets\footnote{The reader may consult Baumgartner's paper for the details and precise definitions concerning this framework.}. The corresponding $\mathcal{I}$-ultrafilters are known as discrete-, scattered-, measure zero- and nowhere dense ultrafilters, respectively. Note that all of these ultrafilter classes contain the $P$-points.

Apart from the simple combinatorial characterisation of Laver ultrafilters given in Definition \ref{def:laver} and Theorem \ref{thm:laverprop}, we will characterise Laver ultrafilters as $\mathcal{I}$-ultrafilters for a certain class of ideals $\mathcal{I}$ on $2^\omega$. We will later use the fact that one may equivalently consider these ideals on the countable subspace $X=\mathbb{Q}\subseteq 2^\omega$ of binary sequences that are eventually zero --  aligning with the work on $\mathcal{I}$-ultrafilters of Brendle \cite{brendle-between}, Barney \cite{barney-uf}, and Brendle-Fla\v{s}kov\'a \cite{brefla17}.

By the aforementioned result of B{\l}aszczyk and Shelah \cite{blaszczyk2001regular}, Laver ultrafilters must in particular be nowhere dense. We show that they in fact have the stronger property of being measure zero ultrafilters. In fact, for each so-called \textit{Yorioka ideal} $\mathcal{Y}_f$, Laver ultrafilters are $\mathcal{Y}_f^0$-ultrafilters, where $\mathcal{Y}_f^0$ consists of those subsets of $2^\omega$ with closure in $\mathcal{Y}_f$. Yorioka ideals are approximations of the strong measure zero ideal $\mathcal{SN}$, that is $\mathcal{SN}:=\bigcap_{f\in \omega^\omega}\mathcal{Y}_f$.

A natural question that might arise is whether Laver ultrafilters are $\mathcal{I}$-ultrafilters for the ideal $\mathcal{I}$ consisting of the sets with closure of strong measure zero. Note however that a closed set has strong measure zero if and only if it is countable, hence the above ideal simply consists of the sets with countable closure. The corresponding $\mathcal{I}$-ultrafilters are known as \textit{countable closed ultrafilters} and were previously investigated by Barney \cite{barney-uf} and Brendle \cite{brendle-between}. We show that a fragment of $\textsf{MA}$ implies the existence of a Laver ultrafilter $\mathcal{U}$ that is non-scattered, which in particular implies that $\mathcal{U}$ is not countable closed. The reader may consult Figure 2 in \cite{brefla17} for a diagram of provable inclusions among various classes of $\mathcal{I}$-ultrafilters.

Finally, we will consider the (generic) existence of Laver ultrafilters. A class of ultrafilters $\mathcal{C}$ is said to \emph{exist generically} if every filter base of cardinality strictly less than $\mathfrak{c}$ can be extended to an ultrafilter in the class $\mathcal{C}$ (c.f. Canjar \cite{Canjar-gen}). Generic existence of an ultrafilter class is generally a more well-behaved phenomenon than mere existence, in the sense that it typically results from straightforward recursive constructions, made possible by certain (in)equalities between cardinal characteristics of the continuum. We prove upper- and lower bounds on the characteristic $\mathfrak{ge}(\text{Laver})$, which is defined to capture the generic existence of Laver ultrafilters. More specifically, we show that $\text{cov}(\mathcal{M}), \text{non}(\mathcal{NA}) \leq \mathfrak{ge}(\text{Laver})\leq \text{non}(\mathcal{SN}), \max\{\text{non}(\mathcal{E}), \mathfrak{d}\}$.\footnote{These cardinal characteristics will be defined in Section \ref{sec:existence}.}

Furthermore, we examine the existence of Laver ultrafilters in various standard models of set theory. In particular, we are interested in the relationship between the existence of $P$-points and the existence of Laver ultrafilters. By our Proposition \ref{prp:rapid}, Laver ultrafilters do not exist in models without rapid ultrafilters, such as the Laver-, Mathias- or Miller models. In each of these, $P$-points exist generically.\footnote{In every model with $\mathfrak{c}\leq \omega_2$, there either exist rapid ultrafilters or $P$-points exist generically. This follows from results of Canjar \cite{Canjar-gen} and Ketonen \cite{ketonen1976existence}.} We show that Laver ultrafilters do not exist in the Silver model, which was shown not to contain $P$-points by Chodounsk\'y and Guzm\'an \cite{chodguz-ppoints}. On the other hand, we show that it is also consistent that $P$-points do not exist while Laver ultrafilters exist generically.

The paper is organized as follows.
\begin{itemize}
    \item In section \S\ref{section2} we give a purely combinatorial definition of Laver ultrafilters and obtain several alternative characterisations. We then show that these indeed define the class of ultrafilters for which the associated Laver forcing has the Laver property.
    \item In section \S\ref{section3}, we present various properties of Laver ultrafilters, such as closure under ultrafilter sums, and establish connections to other well-studied classes of ultrafilters. Under $\textsf{MA}(\text{$\sigma$-linked})$, we construct a Laver ultrafilter that is non-scattered.
    \item In the final section \S\ref{sec:existence}, we investigate the (generic) existence of Laver ultrafilters, define their generic existence number $\mathfrak{ge}(\text{Laver})$ and prove upper- and lower bounds on this cardinal characteristic. We examine the existence of Laver ultrafilters in various well-known models of \textsf{ZFC}, and construct a model without $P$-points in which Laver ultrafilter exist generically. We conclude with some open problems.
\end{itemize}

\subsection*{Notation and Terminology}
We use standard set theoretical notation. For a set $A$ and a cardinal $\kappa$, $[A]^{\kappa}$ denotes the set of subsets of $A$ of cardinality $\kappa$, and $A^\kappa$  the set of functions from $\kappa$ to $A$. In some cases, when the context is clear, we will write $2^k$ to mean $|2^k|$. We denote by $\seq^{<\omega}(\omega)$ the set of finite strictly increasing sequences of natural numbers. The set $\mathbb{Q}$ consists of those $x\in 2^\omega$ with $\exists n \forall m \geq n: x(m)=0$. A basis for the topology on $X=2^\omega$ or its subspace $X=\mathbb{Q}$ is given by the sets $[s]:=\{x \in X: x|_{|s|}=s\}$ for $s\in 2^{<\omega}$.

A \emph{tree} is a nonempty set $T$ of finite sequences that is closed under taking initial segments. As is standard, $\operatorname{stem}(T)$ denotes the longest element of $T$ that is \emph{compatible} with every member of $T$, where $s,t\in T$ are compatible if one extends the other. For $s\in T$, $T|_s:=\{t\in T : \; \text{$s$ and $t$ are compatible}\}$.


If $\mathcal{I}$ is an ideal on a set $X$, then $\mathcal{I}^{+}=\mathcal{P}(X)\setminus \mathcal{I}$. By an \emph{ultrafilter}, we always mean a non-principal ultrafilter. 

\section{Characterisations of Laver Ultrafilters}\label{section2}

We will begin by defining Laver ultrafilters in purely combinatorial terms. Theorem \ref{thm:laverprop} shows that this definition indeed captures the Laver property of $\mathbb{L}_{\mathcal{U}}$.

\begin{defi}\label{def:laver}
Let $\mathcal{U}$ be an ultrafilter over $\omega$. We say that $\mathcal{U}$ is a \textit{Laver ultrafilter} if the following holds: For every sequence $\langle \mathcal{P}_n: n \in \omega\rangle$, where each $\mathcal{P}_n$ is a partition of $\omega$ into finitely many sets, there exists $x \in \mathcal{U}$ such that for all $n \in \omega$, $x$ has non-empty intersection with at most $n+1$ elements of $\mathcal{P}_n$.
\end{defi}

Analogously to the definition of the Laver property, the bound $n+1$ above can be replaced by $h(n)+1$ for any non-decreasing, unbounded $h\in \omega^\omega$.

\begin{fct}\label{fct:bound}
	Let $\mathcal{U}$ be an ultrafilter and let $h_0$ and $h_1$ be two non-decreasing, unbounded functions from $\omega$ to $\omega$. Assume that for any sequence $\langle \mathcal{P}_n: n \in \omega\rangle$ of finite partitions of $\omega$, there exists $x \in \mathcal{U}$ such that $x$ meets at most $h_0(n)+1$ elements of each $\mathcal{P}_n$. Then, for any $\langle \mathcal{P}_n: n \in \omega\rangle$ as above, there exists $y \in \mathcal{U}$ such that $y$ meets at most $h_1(n)+1$ elements of each $\mathcal{P}_n$.
\end{fct}

\begin{proof}
For each $n \in \omega$, let $k_n \geq n$ be such that $\forall m \geq k_n: h_1(m) > h_0(n)$. Let $\langle \mathcal{P}_n: n \in \omega\rangle$ be a sequence of finite partitions of $\omega$. Without loss of generality, we may assume that each $\mathcal{P}_{n+1}$ refines $\mathcal{P}_n$.

Define ${\mathcal{P}_i}':=\mathcal{P}_{k_{i+1}}$ for each $i \in \omega$. By assumption, there exists $x \in \mathcal{U}$ such that $x$ meets at most $h_0(i)+1$ elements of each $\mathcal{P}_i'$. Let $y$ be the intersection of $x$ with the unique element of $\mathcal{P}_{k_0}$ that lies in $\mathcal{U}$. Hence, for any $n \leq k_0$, $y$ meets exactly one element of $\mathcal{P}_n$. If $n > k_0$, then there is a unique $i \in \omega$ such that $n \in [k_i, k_{i+1})$. Since $n \geq k_i$, $h_1(n)> h_0(i)$, and therefore, $y$ meets at most $h_1(n)$ elements of $\mathcal{P}_{i}'=\mathcal{P}_{k_{i+1}}$. Since $\mathcal{P}_n$ is coarser than $\mathcal{P}_{k_{i+1}}$, $y$ meets at most $h_1(n)$ elements of $\mathcal{P}_n$.
\end{proof}

Next, we give an alternative characterisation of Laver ultrafilters, which will allow us to place them in Baumgartner's $\mathcal{I}$-ultrafilter framework.

\begin{defi}
For $A \subseteq 2^\omega$ and $n\in \omega$, define
\[\level_A(n):=|\{f|_n: f \in A\}|.\]	
\end{defi}

\begin{lem}\label{lem:Ff}
$\mathcal{U}$ is a Laver ultrafilter if and only if for any $F:\omega \to 2^\omega$ and any non-decreasing, unbounded $f: \omega \to \omega$, there exists $x \in \mathcal{U}$ such that 
\[\forall n \in \omega: \level_{F[x]}(n)\leq f(n)+1.\]
\end{lem}

\begin{proof}
	Assume that $\mathcal{U}$ is a Laver ultrafilter and that $F,f$ are as in the statement of the lemma. For each $n \in \omega$ let $\mathcal{P}_n:=\{P_s: s \in 2^n\}$, where $P_s=\{i \in \omega: F(i)|_n=s\}$. By Fact \ref{fct:bound}, there exists $x \in \mathcal{U}$ such that $x$ meets at most $f(n)+1$ elements of each $\mathcal{P}_n$, i.e., $\level_{F[x]}(n)\leq f(n)+1$.
	
	To check the reverse direction, let $\langle\mathcal{P}_n: n \in \omega \rangle$ be any sequence of finite partitions of $\omega$. Without loss of generality, we may assume that the cardinality of each $\mathcal{P}_n$ is $2^{k_n}$ for some $k_n > 0$. Enumerate $\mathcal{P}_n$ as 
	\[\mathcal{P}_n=\{P_n^s: s \in 2^{k_n}\},\]
For $i,n \in \omega$, let $s_n^i \in 2^{k_n}$ be such that $i \in P_n^{s_n^i}$. Define $F: \omega \to 2^\omega$ by setting
\[F(i):=(s_0^i) {^{\smallfrown}} (s_1^i) {^{\smallfrown}} (s_2^i) {^{\smallfrown}} ...,\]
and let $f: \omega \to \omega$ be given by $f|_{[0,k_0]}\equiv 0$ and for each $n > 0$
\[f|_{\left(\sum_{n' < n}k_{n'},\; \sum_{n' \leq n}k_{n'}\right]}\equiv n.\]
Now, any $x \in \mathcal{U}$ satisfying $\forall n \in \omega: \level_{F[x]}(\sum_{n'\leq n}k_{n'})\leq f(\sum_{n'\leq n}k_{n'})+1=n+1$ will intersect at most $n+1$ elements of $\mathcal{P}_n$.
\end{proof}

Lemma \ref{lem:Ff} allows us to characterise Laver ultrafilters in terms of ideals on $2^\omega$.

\begin{defi}
Write
\[\mathcal{H}:=\{f \in \omega^{\omega}: f \text{ is non-decreasing, unbounded and } \forall n \in \omega: f(n)\leq n\}.\]
For each $f \in \mathcal{H}$, define 
\[\mathcal{I}_f:=\{A \subseteq 2^\omega: \forall d>0: \level_A \leq^{\ast} \frac{1}{d} \,f\}.\]
It is not hard to check that each such $\mathcal{I}_f$ is an ideal.
\end{defi}

\begin{prp}\label{prp:baumgartnerstyle}
	$\mathcal{U}$ is Laver if and only if $\mathcal{U}$ is an $\mathcal{I}_f$-ultrafilter for each $f \in \mathcal{H}$.
\end{prp}

\begin{proof}
If $\mathcal{U}$ is Laver and $f \in \mathcal{H}$, then $\mathcal{U}$ is an $\mathcal{I}_f$-ultrafilter by applying Lemma \ref{lem:Ff} to some non-decreasing, unbounded $g \in \omega^{\omega}$ such that $\forall d >0: g+1 \leq^{\ast} \frac{1}{d}f$. 

The other direction follows directly from Lemma \ref{lem:Ff}, since $f$  dominating $\level_{F[x]}$ above some $m \in \omega$ implies that $f+1$ dominates $\level_{F[x \cap F^{-1}([s])]}$ everywhere, where $s \in 2^m$ is such that $F^{-1}([s])\in \mathcal{U}$.
\end{proof}

 As mentioned in the introduction, in both Lemma \ref{lem:Ff} and in Proposition \ref{prp:baumgartnerstyle} we may restrict our attention to functions $F: \omega \to \mathbb{Q}$ and consider the ideals $\mathcal{I}_f$ not on $2^\omega$, but on $\mathbb{Q}$. We will discuss this point in Section \ref{sec:existence}, where it will simplify some proofs.

Next, we show that Laver ultrafilters live up to what they promise: They are precisely those $\mathcal{U}$ for which $\mathbb{L}_{\mathcal{U}}$ has the Laver property. Recall the following fact.

\begin{fct}[see Judah and Shelah {\cite[Theorem 1.7]{Sh:321}}]
Let $\mathcal{U}$ be any ultrafilter. The forcing notion $\mathbb{L}_{\mathcal{U}}$ has the pure decision property, i.e., for any sentence $\psi$ in the forcing language, there exists a $T' \leq_{\mathbb{L}_{\mathcal{U}}} T$ with $\text{stem}(T')=\text{stem}(T)$ such that either $T' \Vdash_{\mathbb{L}_{\mathcal{U}}} \psi$ or $T' \Vdash_{\mathbb{L}_{\mathcal{U}}} \neg\psi$.	
\end{fct}

\begin{thm}\label{thm:laverprop}
	$\mathbb{L}_{\mathcal{U}}$ has the Laver property if and only if $\mathcal{U}$ is a Laver ultrafilter.
\end{thm}

\begin{proof}
	Assume first that $\mathcal{U}$ is Laver. Let $\undertilde{f}$ be a $\mathbb{L}_{\mathcal{U}}$-name for an element of $\omega^{\omega}$, let $g \in \omega^{\omega}$ be such that $\mathbb{L}_{\mathcal{U}}\Vdash \forall n \in \omega: \undertilde{f}(n)< g(n)$, and let $T \in \mathbb{L}_{\mathcal{U}}$. We will show that there exists some $T' \leq_{\mathbb{L}_{\mathcal{U}}}T$ and a function $c: \omega \to [\omega]^{<\omega}$ such that $\forall n \in \omega: |c(n)|\leq n^{3} \land T' \Vdash_{\mathbb{L}_{\mathcal{U}}}\undertilde{f}(n)\in c(n)$.

Assume without loss of generality that $\text{stem}(T)=\emptyset$. By the pure decision property, we may thin out $T$ such that
\[\forall s \in T: T|_s \text{ decides } \undertilde{f}(n) \text{ for each } n \leq |s|.\]

By induction, for each $s\in T$, we define:

\begin{enumerate}[label=(\roman*)]
\item A sequence $\langle \mathcal{P}_s^{n}: n \in \omega\rangle$ of partitions of $\omega$ into finitely many pieces,
\item A decreasing sequence $\langle A_s^{n}: n \in \omega\rangle$ of elements of $\mathcal{U}$, such that each $A_s^n$ is one of the elements of $\mathcal{P}_s^n$.
\item If $|s| \leq n$, then some $a_s^n \in g(n)$.
\end{enumerate}

For each $s \in T$, let $A_s^0:=\text{succ}_{T}(s)$ and $\mathcal{P}_s^0$ the partition of $\omega$ into $A_s^0$ and its complement. Furthermore, let $a_{\emptyset}^0 \in g(0)$ such that $T \Vdash_{\mathbb{L}_{\mathcal{U}}}\undertilde{f}(n)=a_{\emptyset}^0$. In the $(n+1)$'th step, define

\begin{enumerate}[label=(\roman*)]
\item For each $s\in T$ with $|s|\geq n+1: A_s^{n+1}:=A_s^{n}$ and $\mathcal{P}_s^{n+1}:=\mathcal{P}_s^{n}$.
\item For each $s \in T$ with $|s|=n+1$, let $a_s^{n+1}\in g(n+1)$ be such that $T|_s \Vdash_{\mathbb{L}_{\mathcal{U}}}\undertilde{f}(n+1)=a_s^{n+1}$.
\item Now, for each $t \in T$ with $|t|=n$, let
\[\mathcal{Q}_t^{n+1}:=\{\{m \in \text{succ}_T(t): a_{t^{\smallfrown}m}^{n+1}=a\}: a \in g(n+1)\} \cup \{\omega \setminus \text{succ}_T(t)\}.\]
and let $\mathcal{P}_t^{n+1}$ be a refinement of $\mathcal{Q}_t^{n+1}$ and $\mathcal{P}_t^{n}$. Let $A_t^{n+1}\in \mathcal{U}$ be the unique element of $\mathcal{P}_t^{n+1} $ that lies in $\mathcal{U}$, and let $a \in g(n+1)$ be such that $\forall m \in A_t^{n+1}: a_{t^{\smallfrown}m}^{n+1}=a$. Set $a_t^{n+1}:=a$ and continue this construction downward.
\end{enumerate}

Let $\varphi: \seq^{<\omega}(\omega) \to \omega$ be a bijection. For each $s\in T$, let $B_s \in \mathcal{U}$ be given by Definition \ref{def:laver} for the sequence $\langle \mathcal{P}_s^{\varphi(s)+n}: n \in \omega\rangle$.

Let $T' \leq_{\mathbb{L}_{\mathcal{U}}}T$ be such that $\forall s \in T': \text{succ}_{T'}(s)=B_s$. Furthermore, define $c(n):=\{a_s^{n}: s \in T' \land|s|=n\}$ for each $n \in \omega$. It is clear that $T' \Vdash_{\mathbb{L}_{\mathcal{U}}}\forall n \in \omega: \undertilde{f}(n)\in c(n)$.

\begin{claim}
$\forall n \in \omega: |c(n)| \leq n^{3}$.
\end{claim}

\begin{proof}
For each $k \leq n$, define $c_k(n):=\{a_s^{n}: s \in T' \land |s|=k\}$, such that $c(n)=c_n(n)$. Let $s \in T'$ with $|s|=k$ and note the following:

If $\varphi(s)\geq n$, then $\mathcal{P}_s^{n}$ is coarser then $\mathcal{P}_s^{\varphi(n)+0}$ and therefore $B_s \subseteq A_s^n$. Hence, $a_{s^{\smallfrown}m}^n = a_s^{n}$ for each $m \in B_s=\text{succ}_{T'}(s)$.

If $\varphi(s)<n$, we know that $B_s$ intersects at most $n+1$ elements of $\mathcal{P}_s^n$, one of which is $A_s^n$. Therefore, among all $m \in \text{succ}_{T'}(s)$, $a_{s^{\smallfrown}m}^n$ attains the value $a_s^n$ and at most $n$ other values. This case happens at most $n$ times, since there are at most $n$ sequences $s$ with $\varphi(s)<n$.

It follows that $c_{k+1}(n)$ is the union of $c_k(n)$ with some set of size at most $n^2$. Hence, $|c_n(n)|\leq n^3$.
\end{proof}

For the reverse direction, assume that $\mathcal{U}$ is not Laver. By Fact \ref{fct:bound}, there exists a sequence $\langle \mathcal{P}_n: n \in \omega \rangle$ of finite partitions of $\omega$ such that
\[\forall x \in \mathcal{U}\: \forall m \in \omega \: \exists n \geq m: ``x \text{ meets at least }|2^n| \text{ elements of }\mathcal{P}_n."\]
Write $\mathcal{P}_n=\{P_n^{i}: i \in k_n\}$. We denote by $\mathbbm{1}$ the maximal element of $\mathbb{L}_{\mathcal{U}}$. Let $\varphi: \seq ^{<\omega}(\omega)\to \omega$ again be a bijection and define $\undertilde{f}$ to be an $\mathbb{L}_{\mathcal{U}}$-name for an element of $\omega^{\omega}$ such that
\[\forall s \in \mathbbm{1}: \mathbbm{1}|_s \Vdash \forall n \leq |s|: \undertilde{f}(n)=\varphi(\sigma(s|_{n})),\]
where $\sigma(t)\in \omega^{|t|}$ is given by
\[\sigma(t)(m)=l \text{ for the unique } l \in k_{|t|-m-1} \text{ with }t(m) \in P_{|t|-m-1}^l.\]
Note that for each $n \in \omega$ and each $t \in \mathbbm{1}$ with $|t|=n$, $\sigma(t)$ is an element of the finite set $\prod_{m \in n}k_{n-m-1}$, hence $\undertilde{f}$ is bounded.

To show that $\undertilde{f}$ constitutes a counterexample to the Laver property for $\mathbb{L}_{\mathcal{U}}$, assume by contradiction that there exists a function $c: \omega \to [\omega]^{<\omega}$ with $\forall n \in \omega: |c(n)|\leq n+1$ and some $T \in \mathbb{L}_{\mathcal{U}}$ such that
\[T \Vdash \forall n \in \omega: \undertilde{f}(n)\in c(n).\]
Let $s:=\text{stem}(T)$ and let $i \in \omega$ be large enough so that
\begin{itemize}
\item[(i)] $2^i - i > |s|+2$
\item[(ii)] $\text{succ}_T(s)$ intersects at least $2^i$ elements of $\mathcal{P}_i$.
\end{itemize}
For each $j \in 2^i$, choose some $a_j \in \text{succ}_T(s)$ such that distinct $a_j$ belong to distinct elements of $\mathcal{P}_i$. Furthermore, for each $j \in 2^i$, let $t_j \in T|_s$ be an extension of $s^{\smallfrown}a_j$ with $|t_j|=|s|+i+1$.

Now, each $T|_{t_j}$ extends $T|_s$ and
\[T|_{t_j} \Vdash \undertilde{f}(|s|+i+1)=\varphi(\sigma(t_j)),\]
where $\sigma(t_j)(|s|)=l_j$ for the unique $l_j \in k_{i}$ such that $t_j(|s|)=a_j\in P_i^{l_j}$. Since these $l_j$ are distinct for distinct $j$, the $T|_{t_j}$ decide $\undertilde{f}(|s|+i+1)$ in $2^i$ different ways. Hence, since $2^i > |s|+i+2\geq |c(|s|+i+1)|$, there must be some $j \in 2^i$ such that $T|_{t_j}\Vdash \undertilde{f}(|s|+i+1)\notin c(|s|+i+1)$.
\end{proof}

\section{Properties of Laver Ultrafilters}\label{section3}

In this section, we establish several structural properties of Laver ultrafilters. We also relate them to other well-studied classes of ultrafilters and, assuming $\textsf{MA}(\sigma\text{-linked})$, construct a non-scattered Laver ultrafilter.

Recall that an ultrafilter $\mathcal{V}$ is \emph{Rudin-Keisler below} an ultrafilter $\mathcal{U}$, denoted $\mathcal{V}\leq_{\text{RK}}\mathcal{}U$, if there exists some $f\in \omega^\omega$ such that $\forall x \in \mathcal{V}: f^{-1}(x)\in \mathcal{U}$. Note the following fact.

\begin{fct}\label{fct:starhereditary}
If $\mathcal{U}$ is a Laver ultrafilter and $\mathcal{V}\leq_{\text{RK}}\mathcal{U}$, then $\mathcal{V}$ is a Laver ultrafilter.
\end{fct}

\begin{prp}\label{prp:rapid}
If $\mathcal{U}$ is a Laver ultrafilter, then $\mathcal{U}$ is hereditarily rapid, i.e., if $\mathcal{V}$ is an ultrafilter such that $\mathcal{V}\leq_{\text{RK}}\mathcal{U}$, then $\mathcal{V}$ is rapid.
\end{prp}

\begin{proof}
For a strictly increasing $f \in \omega^\omega$, consider the sequence $\langle \mathcal{P}_n: n \in \omega\rangle$ of finite partitions of $\omega$, where
\[\mathcal{P}_n=\{\{i\}: i < f(n)\} \cup \{\omega \setminus f(n)\}.\]
Since the property of being a Laver ultrafilter is downward closed in the Rudin-Keisler ordering by Fact \ref{fct:starhereditary}, the claim follows.
\end{proof}

Note that Brendle and Fla\v{s}kov\'a \cite[Theorem 2.8]{brefla17} have constructed a hereditarily rapid ultrafilter that is not nowhere dense, assuming $\textsf{MA}(\text{countable})$.\footnote{In fact, the ultrafilter in question is a hereditary $Q$-point.} Therefore, as Laver ultrafilters are nowhere dense, the implication in Proposition \ref{prp:rapid} does not reverse. If we additionally assume that our hereditarily rapid ultrafilter is a $P$-point, then we do obtain a Laver ultrafilter.\footnote{Note that a rapid $P$-point $\mathcal{U}$ is automatically hereditarily rapid, since every Rudin-Keisler reduction $f\in \omega^\omega$ is finite-to-one on some element of $\mathcal{U}$.}

\begin{prp}\label{prp:semiselective}
Every rapid $P$-point is a Laver ultrafilter.
\end{prp}

\begin{proof}
Let $\mathcal{U}$ be a rapid $P$-point and consider any sequence $\langle \mathcal{P}_n: n \in \omega\rangle$ of finite partitions of $\omega$. Denote by $u(n)$ the unique element in $\mathcal{P}_n \cap \mathcal{U}$. Since $\mathcal{U}$ is a $P$-point, the sequence $\langle u(n): n \in \omega \rangle$ has a pseudo-intersection $x$ in $\mathcal{U}$, and since $\mathcal{U}$ is rapid, we may assume that $\forall n \in \omega: |x \setminus u(n)|\leq n$. Hence, apart from $u(n)$, $x$ meets at most $n$ additional elements of $\mathcal{P}_n$.
\end{proof}

Given ultrafilters $\mathcal{U}$ and $\mathcal{V}_i$, $i\in \omega$, recall that the ultrafilter $\mathcal{U}-\sum_{i\in \omega}\mathcal{V}_i$ on $\omega \times \omega$ consists of those $x\subseteq \omega \times \omega$ for which
\[\{i \in \omega: \{j\in \omega: \langle i,j\rangle\in x\}\in \mathcal{V}_i\}\in \mathcal{U}.\]
Alternatively, $\mathcal{U}-\sum_{i\in \omega}\mathcal{V}_i$ is the $\mathcal{U}$-limit of the sequence $\langle \mathcal{V}_i: i \in \omega\rangle$ in the topological space $\beta \omega \setminus \omega$.

\begin{prp}
If $\mathcal{U}$ and $\mathcal{V}_i$ for  $i \in \omega$ are Laver ultrafilters, then ${\mathcal{U}}-\sum_{i \in \omega}\mathcal{V}_i$ is a Laver ultrafilter as well.
\end{prp}

\begin{proof}
 Let $\langle \mathcal{P}_n: n \in \omega\rangle$ be a sequence of finite partitions of $\omega \times \omega$. By Fact \ref{fct:bound}, it suffices to find some $X \in {\mathcal{U}}-\sum_{i\in \omega}\mathcal{V}_i$ such that for each $n \in \omega$, $X$ intersects at most $(n+1)^{2}$ elements of $\mathcal{P}_n$. We may assume that $\mathcal{P}_{n+1}$ refines $\mathcal{P}_n$.
 
If $A \subseteq \omega \times \omega$ and $i \in \omega$, write $A^{(i)}:=\{j \in \omega: \langle i,j\rangle \in A\}$. For $i,n \in \omega$, let $\mathcal{P}_n^{(i)}$ be the partition $\{P^{(i)}: P \in \mathcal{P}_n\}$. Furthermore, let $Q_n(i)$ be the unique element of $\mathcal{P}_n$ with $Q_n(i)^{(i)}\in \mathcal{V}_i$. Define the partition
 \[\mathcal{Q}_n:=\{\{i \in \omega: Q_n(i)=P\}: P \in \mathcal{P}_n\}.\]
 For each $i \in \omega$, let $y_i\in \mathcal{V}_i$ be such that for each $n\in \omega$, $y_i$ intersects at most $n+1$ elements of $\mathcal{P}_{n+i}^{(i)}$. Let $x \in \mathcal{U}$ be such that $x$ intersects at most $n+1$ elements of $\mathcal{Q}_n$. Consider the set
 \[X:=\{\langle i,j\rangle: i \in x \land j \in y_i\}\in \mathcal{U}-\sum_{i \in \omega}\mathcal{V}_i.\]
 Fix $n \in \omega$ and consider the decomposition $X=X_{<n} \cup X_{\geq n}$, where $X_{<n}:=X\cap (n \times \omega)$ and $X_{\geq n}:= X \cap ((\omega \setminus n)\times \omega)$. Note that $X_{<n}$ intersects at most $n \cdot (n+1)$ elements of $\mathcal{P}_n$. Furthermore, for each $i \geq n$, $X_{\geq n}^{(i)}=y_i$ intersects at most one element of $\mathcal{P}_n^{(i)}$, hence $X_{\geq n}^{(i)} \subseteq Q_n(i)^{(i)}$. Therefore, since $x$ intersects at most $n+1$ elements of $\mathcal{Q}_n$, $X_{\geq n}$ intersects at most $n+1$ elements of $\mathcal{P}_n$. It follows that $X$ intersects at most
 \[n \cdot (n+1)+(n+1)=(n+1)^{2}\]
 elements of $\mathcal{P}_n$.
 \end{proof}

\begin{cor}
If there exists a Laver ultrafilter, there exists a Laver ultrafilter that is not a $P$-point.
\end{cor}

\begin{proof}
Given a Laver ultrafilter $\mathcal{U}$, the ultrafilter $\mathcal{U}-\sum_{n\in \omega}\mathcal{U}=\mathcal{U}\times \mathcal{U}$ is a Laver ultrafilter by the previous proposition. It is easy to see that the product of two ultrafilters is never a $P$-point.
\end{proof}

While Laver ultrafilters must not be $P$-points, they nonetheless share some $P$-point-like properties. More concretely, as mentioned in the introduction, they are measure zero ultrafilters. We need the following definition due to Yorioka \cite{yoriokaideals}.

\begin{defi}
For $\sigma \in (2^{<\omega})^{\omega}$,	 define $[\sigma]_{\infty}:=\bigcap_{n \in \omega}\bigcup_{m \geq n}[\sigma(m)]$ and $\text{ht}_{\sigma}(i):=|\sigma(i)|$, for each $i \in \omega$. For $f,g \in \omega^\omega$, define $f \ll g :\iff \forall k \in \omega: f \circ \text{pw}_k \leq^{\ast}g$, where $\text{pw}_k: \omega \to \omega, i \mapsto i^k$. Finally, for an increasing $f \in \omega^\omega$, define the \textit{Yorioka ideal}
\[\mathcal{Y}_f:=\{X \subseteq 2^{\omega}: \exists \sigma \in (2^{<\omega})^{\omega}: X \subseteq [\sigma]_{\infty} \land f \ll \text{ht}_{\sigma}\}.\]
\end{defi}

Yorioka \cite{yoriokaideals} has shown that for each strictly increasing $f \in \omega^\omega$, $\mathcal{Y}_f$ is a $\sigma$-ideal, $\mathcal{Y}_f$ consists of measure zero sets and $\bigcap_{f\in \omega^\omega}\mathcal{Y}_f$ is precisely the ideal $\mathcal{SN}$ of strong measure zero subsets of $2^\omega$.

\begin{lem}\label{lemyor}
For each strictly increasing $f \in \omega^\omega$, there exists $g\in\mathcal{H}$ such that $\mathcal{I}_g \subseteq \mathcal{Y}_f$.
\end{lem}

\begin{proof}
	Let $f^{\ast}\in \omega^\omega$ be increasing and a $\leq^{\ast}$-upper bound of the family $\{f \circ \text{pw}_k: k \in \omega\}$. Define $h_i:=f^{\ast}(i(i+3))$ for each $i \in \omega$. Let $g \in \omega^\omega$ be given by 
	\[g|_{[h_i, h_{i+1})}:\equiv i \;\text{ and }\; g|_{[0,h_0)}:\equiv 0.\]
	Now, for a given $X \in \mathcal{I}_g$, let $\sigma_X \in (2^{<\omega})^\omega$ enumerate the family $\bigcup_{i \in \omega}\{x|_{h_i}: x \in X\}$, such that $\text{ht}(\sigma_X)$ is non-decreasing.
	
	There exists $m \in \omega$ such that $j \geq m \implies \level_X(h_j) \leq g(h_j)+1=j+1$. Hence, for each $i \geq m$,
	\[\begin{aligned}|\bigcup_{j \leq i} \{x|_{h_j}: x \in X\}|&=\sum_{j \leq i} \level_X(h_j)\\
	&\leq m (m+1)+ \sum_{m \leq j \leq i}(j+1)\\
	&=\frac{m(m+1)}{2}+\frac{i(i+3)}{2}+1=\frac{i(i+3)}{2}+M,\end{aligned}\]
where $M \in \omega$ is a constant. Therefore, if $n,i \in \omega$ are such that $\text{ht}_{\sigma_X}(n)=h_i< f^{\ast}(n)$, then $n < i(i+3)/2 + M$, hence $f^{\ast}(i(i+3)/2 + M)> h_i = f^{\ast}(i(i+3))$, and thus $i(i+3)/2 +M > i(i+3)$, which holds for only finitely many $i$. It follows that $\text{ht}_{\sigma_X}\geq^{\ast}f^{\ast}$.
\end{proof}

Denote by $\mathcal{Y}_f^0$ the ideal consisting of those subsets of $2^\omega$ with closure in $\mathcal{Y}_f$. Since each $\mathcal{I}_g$ is closed under the closure operation, we obtain the following corollary.

\begin{cor}\label{cor:yorioka}
If $\mathcal{U}$ is a Laver ultrafilter, then $\mathcal{U}$ is a $\mathcal{Y}_f^0$-ultrafilter for each Yorioka ideal $\mathcal{Y}_f$. In particular, $\mathcal{U}$ is measure zero (and thus nowhere dense).
\end{cor}

As discussed in the introduction, the countable closed ultrafilters naturally strengthen the property of being a $\mathcal{Y}_f^0$-ultrafilter for each $\mathcal{Y}_f$. Our next result shows that \textsf{MA}($\sigma$-linked) implies the existence of a Laver ultrafilter $\mathcal{U}$ that is not scattered, i.e., not an $\mathcal{I}$-ultrafilter for the ideal $\mathcal{I}$ of scattered subsets of $2^\omega$. The class of scattered ultrafilters contains the countable closed ultrafilters, as well as the discrete- and the \emph{$\sigma$-compact ultrafilters}. Hence, it follows that Corollary \ref{cor:yorioka} captures all the \textsf{ZFC}-inclusions among the classes of Laver ultrafilters and the other $\mathcal{I}$-ultrafilter classes for the ideals $\mathcal{I}$ on $2^\omega$ studied by Baumgartner \cite{baumgartner-uf}, Barney \cite{barney-uf}, and Brendle \cite{brendle-between}.

The core of the construction of the above ultrafilter $\mathcal{U}$ is contained in the following lemma. We will construct $\mathcal{U}$ on the countable set $\mathbb{Q}\subseteq 2^\omega$.

\begin{lem}[\textsf{MA}($\sigma$-linked)]\label{lem:starnonscattered}
	Assume that $\mathcal{X}\subseteq [\mathbb{Q}]^{\omega}, |\mathcal{X}|<\mathfrak{c}$, is a family of non-scattered sets that is closed under finite intersections. Let $F: \mathbb{Q}\to 2^\omega$ and $f \in \omega^{\omega}$ be non-decreasing and unbounded. There exists $A \in [\mathbb{Q}]^\omega$ such that $\level_{F[A]}\leq f+1$ and $A \cap X$
 is non-scattered for each $X \in \mathcal{X}$.
 \end{lem}
 
 \begin{proof}
 Note the well-known fact that for each $A \subseteq 2^\omega$, there is a unique decomposition $A=d(A) \cup (A \setminus d(A))$, where $d(A)$ is dense-in-itself and $A \setminus d(A)$ is scattered. For $X \in \mathcal{X}, s \in 2^{<\omega}$, write
 \[C_{X,s}:=\cls\left(d(X \cap F^{-1}([s]))\right),\]
 i.e., $C_{X,s}$ is the (possibly empty) closure in $2^\omega$ of the dense-in-itself portion of $X \cap F^{-1}([s])$.
 
 We will repeatedly need the following claim.
 
 \begin{claim}\label{claim:extension}
	Assume $X \in \mathcal{X}$, $s \in 2^{<\omega}$ and $a \in C_{X,s}$. For each $n \geq |s|$, there exists $t \in 2^n$, $t \supseteq s$, such that $a \in C_{X,t}$.  
\end{claim}

\begin{proof}
Since $a \in C_{X,s}=\cls\left(d(X \cap F^{-1}([s]))\right) $, there exists for each $m \in \omega$ some $b_m \in d(X \cap F^{-1}([s]))$ such that $b_m|_m = a|_m$. Since $X \cap F^{-1}([s]) \cap [b_m|_m]$ is non-scattered, there exists an element of the finite partition
\[\{F^{-1}([t])\cap X \cap F^{-1}([s]) \cap [b_m|_m]: t \in 2^{n}\}\]
that is non-scattered. Say that this is the case for $t=t_m$. By the pigeonhole principle, there exists some $\bar t \in 2^n$ such that $t_m=\bar t$ for infinitely many $m$. This $\bar t$ is as desired.
\end{proof}
 
 We are now ready to define the partial order $\mathbb{P}$ on which we use \textsf{MA}($\sigma$-linked). Let $\mathbb{P}$ consist of conditions $p=\langle A_p, \mathcal{X}_p, n_p, S_p\rangle$, where
 \[A_p \in [\mathbb{Q}]^{<\omega}, \mathcal{X}_p \in [\mathcal{X}]^{<\omega}, n_p \in \omega \text{ and } S_p \in [2^{n_p}]^{<\omega},\]
 satisfying
 \begin{enumerate}[label=(\roman*)]
 \item $f(n_p)> 2 \cdot |A_p| \cdot |\mathcal{X}_p|\cdot |\mathcal{S}_p|+ |A_p|+3$,
 \item there exists $s \in S_p$ such that $d(X \cap F^{-1}([s]))\neq \emptyset$ for each $X \in \mathcal{X}$,
 \item $A_p \subseteq \bigcup_{s \in S_p}F^{-1}([s])$,
 \item $S_p$ has levels of size $\leq f+1$, i.e., $\forall n \leq n_p: |\{s|_n: s \in S_p\}| \leq f(n)+1$.
 \end{enumerate}

Let $p \leq q$ if and only if $A_p \supseteq A_q$, $\mathcal{X}_p \supseteq \mathcal{X}_q$, $n_p \geq n_q$, $\forall s \in S_p \exists t \in S_q: s \supseteq t$, and
\[(\dagger)\; \forall a \in A_q \forall X \in \mathcal{X}_q \forall s \in S_q: a \in X \cap C_{X,s} \implies \exists t \in S_p: a \in X \cap C_{X,t}.\]

\begin{claim}\label{claim:linked}
	If $q,q' \in \mathbb{P}$ are such that $A_q =A_{q'}$, $n_q=n_{q'}$ and $S_q=S_{q'}$, then $q$ and $q'$ are compatible, hence $\mathbb{P}$ is $\sigma$-linked.
\end{claim}

\begin{proof}
Define $\mathcal{X}_{p}:=\mathcal{X}_{q} \cup \mathcal{X}_{q'}$, $A_p:=A_q=A_{q'}$, let $n_p \geq n_q=n_{q'}$ be such that
\[f(n_p)> 2 \cdot |A_p|\cdot |\mathcal{X}_p| \cdot f(n_q) + |A_p| + 3,\]
and write $S:=S_q=S_{q'}$.

For each $\langle a, X, s\rangle \in A_p \times \mathcal{X}_p \times S$, choose $t(a,X,s)\in 2^{n_p}$, $t(a,X,s) \supseteq s$ such that,
\[\text{if } a \in X \cap C_{X,s}, \text{ then } a \in X \cap C_{X, t(a,X,s)}.\]
This is possible by Claim \ref{claim:extension}.

Furthermore, for the $s_0 \in S$ given by (ii), let $t_0 \in 2^{n_p}$,  $t_0 \supseteq s_0$ be such that (ii) is satisfied for $t_0$. This is possible because one of the elements of the partition $\{F^{-1}([t]): t \in 2^{n_p}\cap [s_0]\}$ of $[s_0]$ must have non-scattered intersection with each $X \in \mathcal{X}$, since $\mathcal{X}$ is closed under finite intersections. Define
\[S_p:=\{t(a,X,s): \langle a, X, s\rangle \in A_p \times \mathcal{X}_p \times S\} \cup \{a|_{n_p}: a \in A_p\} \cup \{t_0\}.\]
Assume without loss of generality that $|\mathcal{X}_q|\geq |\mathcal{X}_{q'}|$. Note that
\[\begin{aligned}
	|S_p|&\leq |A_p| \cdot (|\mathcal{X}_q|+|\mathcal{X}_{q'}|)\cdot |S|+|A_p|+1 \\
	&\leq 2 \cdot |A_q| \cdot |\mathcal{X}_q| \cdot |S_q|+|A_q|+1\\
	&< f(n_{q})=f(n_{q'})
\end{aligned}.\]
It follows that the choice of $n_p$ above was large enough to satisfy (i) and that $S_p$ has levels of size $\leq f+1$. Hence $p:=\langle A_p, \mathcal{X}_p, n_p, S_p\rangle$ is indeed a $\mathbb{P}$-condition and it extends both $q$ and $q'$.
\end{proof}

We want to show that if $G\subseteq \mathbb{P}$ is a $\mathcal{D}$-generic filter for a certain family $\mathcal{D}$ of open dense sets, then $A_G:=\bigcup \{A_p: p \in G\}$ is as desired. Note that by (i) and (iii), $\level_{F[A_p]}\leq f+1$ for each $p \in \mathbb{P}$, hence $\level_{F[A_G]}\leq f+1$ will hold for any filter $G$.

For each $X \in \mathcal{X}$, define
\[D_X:=\{p \in \mathbb{P}: X \in \mathcal{X}_p \land \exists s \in S_p: A_p \cap X \cap C_{X,s}\neq \emptyset\}.\]

\begin{claim}\label{claim:DX}
$D_X$ is open dense.	
\end{claim}

\begin{proof}
Note that by condition $(\dagger)$ in the ordering of $\mathbb{P}$, $D_X$ is open. To see that it is dense, assume $q \notin D_X$. Define $\mathcal{X}_p:=\mathcal{X}_q \cup \{X\}$ and let $n_p \geq n_q$ be such that $f(n_p)> 2 \cdot (|A_q|+1)\cdot (|\mathcal{X}_q|+1) \cdot f(n_q)+(|A_q|+1)+3$. Let $s_0 \in S_q$ be such that $d(X \cap F^{-1}([s_0]))\neq \emptyset$, as guaranteed by (ii). Choose some $a^{\ast} \in d(X \cap F^{-1}([s_0]))$. Since, in particular, $a^{\ast} \in C_{X,s_0}$, we find by Claim \ref{claim:extension} some $t^{\ast} \in 2^{n_p}$, $t^{\ast} \supseteq s_0$ with $a^{\ast} \in C_{X,t^{\ast}}$. Furthermore, we find some $t_0 \in 2^{n_p}$ extending $s_0$ such that (ii) is satisfied for $t_0$.

Finally, choose $t(a,X,s)\in 2^{n_p}$ for each $\langle a, X, s\rangle \in A_q \times \mathcal{X}_q \times S_q$ as in the proof of Claim \ref{claim:linked} and define
\[\begin{aligned}
	A_p&:=A_q \cup \{a^{\ast}\},\\
	S_p&:=\{t(a,X,s): \langle a, X, s\rangle \in A_q \times \mathcal{X}_q \times S_q\} \cup \{a|_{n_p}: a \in A_p\} \cup \{t_0, t^{\ast}\}.
\end{aligned}\]
Note that $|S_p|\leq |A_q|\cdot |\mathcal{X}_q|\cdot |S_q|+ |A_q|+1 +2 < f(n_q)$, hence $n_p$ is large enough to satisfy (i) and $|S_p|$ has levels of size $\leq f+1$. Thus, $p:=\langle A_p, \mathcal{X}_p, n_p, S_p \rangle \in D_X$ and $p \leq q$.
\end{proof}

Finally, define for each $a \in \mathbb{Q}$, $X \in \mathcal{X}$ the set
\[E_{a,X}:=\{p \in \mathbb{P}: X \in \mathcal{X}_p \land \exists s \in S_p:  a \in A_p \cap X \cap C_{X,s}\},\]
 and for $k \in \omega$, define
\[\begin{aligned}D_{a, X, k}:=\{q \in \mathbb{P}: q \text{ is incompatible with every } p \in E_{a,X}, \text { or } \\
X \in \mathcal{X}_q \land \exists t \in S_q \exists a' \in A_q \cap X \cap C_{X,t}: a'|_k = a|_k \land a' \neq a\}.
\end{aligned}\]

\begin{claim}
$D_{a,X,k}$ is open dense.
\end{claim}

\begin{proof}
It again follows from $(\dagger)$ that $D_{a,X,k}$ is open. To check density, assume $q \notin D_{a,X,k}$. In particular, $q$ is compatible with some $p_0 \in E_{a,X}$. Let $q \geq p \leq p_0$. It follows that $X \in X_p$, and that, by $(\dagger)$, there is some $t \in S_{p}$ such that $a \in A_{p}\cap X \cap C_{X,t}$. Let $a' \in d(X \cap F^{-1}([t]))$ be such that $a'|_k=a|_k$ and $a \neq a'$. Define $A_{\bar p}=A_p \cup \{a'\}$, $\mathcal{X}_{\bar p}=\mathcal{X}_p$ and $n_{\bar p}, S_{\bar p}$ as in the proof of Claim \ref{claim:DX}.
\end{proof}

Now, assume $G$ is generic for each $D_X$ and each $D_{a,X,k}$. We claim that for each $X \in \mathcal{X}$, the set 
\[d_X:=\{a \in \mathbb{Q}: G \cap E_{a,X} \neq \emptyset\},\] which is a subset of $A_G \cap X$, is non-empty and dense-in-itself. It is clear that $d_X$ is non-empty, since $G$ intersects $D_X$. To see that it is dense-in-itself, let $a \in d_X$ and $k \in \omega$. There exists $p \in G \cap E_{a,X}$ and $q \in G \cap D_{a,X,k}$. Since $G$ is directed, $p$ and $q$ are compatible, and hence 
\[X \in \mathcal{X}_q \land \exists t \in S_q \exists a' \in A_q \cap X \cap C_{X,t}: a'|_k = a|_k \land a' \neq a.\]
Hence, $q \in E_{a', X}$ and therefore $a' \in d_X$.
\end{proof}

\begin{thm}[\textsf{MA}($\sigma$-linked)]
	There exists a non-scattered Laver ultrafilter.
\end{thm}

\begin{proof}
By transfinite recursion, iterate through all the pairs $\langle F,f\rangle$, where $F: \mathbb{Q} \to 2^\omega$ and $f \in {^{\omega}}\omega$ is non-decreasing and unbounded. At each stage, add a witness $A \in [\mathbb{Q}]^\omega$ to Lemma \ref{lem:Ff} for the pair $\langle F,f\rangle$, using Lemma \ref{lem:starnonscattered}. Finally, extend the resulting filter to an ultrafilter, avoiding the sets that have non-scattered intersection with an element of the previously constructed filter.
\end{proof}

\section{Existence and Generic Existence}\label{sec:existence}

In this final section, we establish results on the (generic) existence of Laver ultrafilters. In particular, we will prove lower- and upper bounds on their \emph{generic existence number}, investigate whether Laver ultrafilters exist in various classical models of \textsf{ZFC}, and obtain a model in which $P$-points do not exist, but Laver ultrafilters exist generically. Throughout this section, by \emph{Cohen model} we mean the model obtained by adding $\kappa\geq\omega_2$-many Cohen reals to a model of \textsf{CH}, and \emph{random model} means the analogous model for random reals. The \emph{Mathias-, Laver-, Miller-, Silver-}, and the \emph{Sacks models} are the models obtained by iterating the corresponding proper forcing notion with countable supports $\omega_2$-many times over a model of \textsf{CH}.

Recall that a class of ultrafilters $\mathcal{C}$ exists generically if every filter base of size strictly less than $\mathfrak{c}$ can be extended to an ultrafilter in the class $\mathcal{C}$. The following key cardinal invariant was introduced by Brendle and Fla\v{s}kov\'a \cite{brefla17} in their study of the generic existence of $\mathcal{I}$-ultrafilters, for various tall ideals $\mathcal{I}$ on a countable set $X$.

\begin{defi}[\cite{brefla17}]
    $$\mathfrak{ge}(\mathcal{I})=\min\{|\mathcal{F}| : \mathcal{F}\subseteq \mathcal{I}^+ \; \text{is a filter base, and} \; \forall I\in\mathcal{I} \; \exists F\in\mathcal{F} \; (|I\cap F|<\omega)\}.$$
\end{defi}

Since we defined the ideals $\mathcal{I}_f$ on the space $2^\omega$, we will first show that restricting the $\mathcal{I}_f$ to $\mathbb{Q}$ results in an equivalent characterisation of Laver ultrafilters.

For distinct $x,y\in 2^\omega$, write $\text{split}(x,y):=\min\{n \in \omega: x(n)\neq y(n)\}$.

\begin{lem}\label{lem:redefinition}
Let $X \subseteq 2^\omega$ be countable. There exists an injection $\phi_X: X \to \mathbb{Q}$ such that $\forall x \neq y \in X: \text{split}(\phi_X(x),\phi_X(y))= \text{split}(x,y)$.	
\end{lem}

It follows that

\begin{lem}
Let $f\in \omega^\omega$ be non-decreasing and unbounded. The following are equivalent:
\begin{enumerate}[label=(\roman*)]
\item For any $F: \omega \to 2^\omega$, there exists $x\in \mathcal{U}$ such that $\level_{F[x]}\leq f+1$.
\item For any $G: \omega \to \mathbb{Q}$, there exists $y\in \mathcal{U}$ such that $\level_{G[y]}\leq f+1$.
\end{enumerate}
\end{lem}

\begin{proof}
It is clear that (i) implies (ii). For the reverse direction, consider the function $\phi_{F[\omega]}\circ F: \omega \to \mathbb{Q}$.
\end{proof}

\begin{proof}[Proof of Lemma \ref{lem:redefinition}]
We will denote $\phi_X$ by $\phi$. Let $\{x_i: i \in \omega\}$ be an enumeration of $X$. Define $\phi(x_0)$ to be the all-zero sequence. Fix $i \in \omega \setminus \{0\}$ and assume by induction that $\phi(x_j)$ is defined for each $j < i$, such that for all distinct $j,j'<i: \text{split}(\phi(x_j),\phi(x_{j'}))= \text{split}(x_j,x_{j'})$. Let $n_i:=\max\{\text{split}(x_j,x_i): j<i\}$ and let $j_{\text{max}} <i$ be such that this maximum is attained at $j_{\text{max}}$. Define
\[\phi(x_i)|_{n_i}:=\phi(x_{j_{\text{max}}})|_{n_i},\; \phi(x_i)(n_i):=1\;\text{and else,}\; \phi(x_i)(m):=0.\]
We leave it to the reader to verify that this $\phi$ works.
\end{proof}

Hence, we will assume in the remainder of the paper that each $\mathcal{I}_f$ is an ideal on the countable set $\mathbb{Q}$.

\begin{defi}
Recall that $\mathcal{H}\subseteq\omega^\omega$ denotes the set of non-decreasing and unbounded functions with $f \leq \text{id}_{\omega}$. Define $$\mathfrak{ge}(\text{Laver}):=\min\{\mathfrak{ge}(\mathcal{I}_f) : f\in\mathcal{H}\}.$$
\end{defi}

The following is the analogue of Observation 3.1 in \cite{brefla17}.

\begin{fct}
The following are equivalent:
\begin{enumerate}[label=(\roman*)]
    \item $\mathfrak{ge}(\text{Laver})=\mathfrak{c}$,
    \item generic existence of Laver ultrafilters.
\end{enumerate}
\end{fct}


As $\text{cov}(\mathcal{M})$ is the Martin number for countable partial orders, it is not hard to see that

\begin{prp}\label{prp:lowerboundcov}
$\text{cov}(\mathcal{M})\leq\mathfrak{ge}(\text{Laver})$.
\end{prp}

\begin{proof}
Let $f\in \mathcal{H}$ and assume that $\mathcal{F} \subseteq [\mathbb{Q}]^\omega$ is a filter base of cardinality $<\text{cov}(\mathcal{M})$. Let $g_f\in \omega^\omega$ with $g_f(0)=0$ be strictly increasing, such that $\forall i \in \omega \;\forall n \in [g_f(i), g_f(i+1)): f(n)\geq n$. Let $\eta \in 2^{\omega}$ be such that $[\eta|_n]\cap \mathbb{Q}$ extends $F$ to a filter base for any $n \in \omega$. For $x\in \mathbb{Q}\setminus \{\eta\}$, define $\text{split}(x)\in \omega$ to be the unique $i \in \omega$ with $\text{split}(x, \eta)\in [g(i), g(i+1))$. Consider the countable partial order
\[\mathbb{P}:=\{X \in [\mathbb{Q}\setminus \{\eta\}]^{<\omega}: \forall x\neq y \in X: \text{split}(x)>0 \land \text{split}(x)\neq \text{split}(y)) \},\]
ordered by inclusion. The obvious choice of open dense sets of $\mathbb{P}$ clearly produces an $A \in [\mathbb{Q}\setminus \{\eta\}]^{\omega}$ extending $\mathcal{F}$ with $\level_A \leq f+1$.
\end{proof}

We will now find an additional lower bound for $\mathfrak{ge}(\text{Laver})$. Let us recall that a set $X\subseteq2^\omega$ is called \emph{null-additive} if $X+N\in\mathcal{N}$ for all $N\in\mathcal{N}$, where $\mathcal{N}$ denotes the ideal of measure zero subsets of $2^\omega$. The characteristic $\text{non}(\mathcal{NA})$ is defined to be the smallest cardinality of a subset of $2^\omega$ that is not null-additive. Recalling Pawlikowski's \cite{paw} characterisation of $\text{non}(\mathcal{NA})$ in terms of slaloms, we now introduce some notation from \cite{carmej}:

\begin{defi}
For a sequence of non-empty sets $b=\langle b(n): n \in \omega\rangle$ and $h \in \omega^\omega$, let
\[\begin{aligned}\prod b&:=\prod_{n \in \omega}b(n),\; \text{and}\;
\mathcal{S}(b,h):=\prod_{n \in \omega}[b(n)]^{\leq h(n)}. \end{aligned}\]	For $x\in \prod b$ and $S \in \mathcal{S}(b,h)$, write $x \in^{*}S$ iff $x(n)\in S(n)$ for all but finitely many $n\in \omega$. Finally, define
\[\mathfrak{b}^{\text{Lc}}_{b,h}:=\min\{|F|: F \subseteq \prod b \land \neg \exists S \in \mathcal{S}(b,h)\; \forall x \in F: x \in^{*}S \}.\]
\end{defi}

Pawlikowski's characterisation and its slight modification from \cite{carmej} state: 

\begin{lem}[\cite{paw}, Lemma 2.2; \cite{carmej}, Lemma 3.9]\label{lem-non(na)char}
For any $h \in \omega^\omega$ that diverges to infinity,
\[\text{non}(\mathcal{NA})=\min\{\mathfrak{b}^{\text{Lc}}_{b,h}: b \in \omega^\omega\}.\]	
\end{lem}

\begin{prp}\label{prp:lowerboundnon}
    $\text{non}(\mathcal{NA})\leq\mathfrak{ge}(\text{Laver})$.
\end{prp}

\begin{proof}
Let $f\in \mathcal{H}$ and assume that $\mathcal{F} \subseteq [\mathbb{Q}]^\omega$ is a filter base of cardinality $<\text{non}(\mathcal{NA})$. It suffices to find some $A \subseteq \mathbb{Q}$ with $\level_A\leq f+1$ such that $A$ has non-empty intersection with each member of $\mathcal{F}$, since we may assume that $\mathcal{F}$ contains the cofinite subsets of $\mathbb{Q}$.

Fix some non-decreasing, unbounded $g \in \omega^\omega$ such that $\forall n \in \omega: g(n)(g(n)+1) \leq f(n)$. Let $\eta \in 2^{\omega}$ be such that $[\eta|_n]\cap \mathbb{Q}$ extends $\mathcal{F}$ to a filter base for every $n\in \omega$. For each $X \in \mathcal{F}$ and $k \in \omega$, choose $y_k^{X} \in X$ such that $y_k^{X}|_n=\eta|_n$ for every $n \in \omega$ with $g(n)\leq k$, and define
\[y^{X}(n):=\{y^{X}_k|_n: k \in \omega\}.\]
Since $y^{X}_k|_n=\eta|_n$ for every $k \geq g(n)$, $y^{X}(n)$ has size at most $g(n)+1$.

Define $b(n):=[2^n]^{\leq g(n)+1}$ for each $n \in \omega$. Hence, $y^{X}\in \prod b$ for each $X \in \mathcal{F}$. Since $|\mathcal{F}|<\mathfrak{b}^{\text{Lc}}_{b,g}$, we find $S \in \mathcal{S}(b, g)$ such that $y^{X}\in^{\ast}S$ for every $X \in \mathcal{F}$. Define
\[A:=\{x \in \mathbb{Q}: \exists n_0 \in \omega : x|_{n_0}=\eta|_{n_0} \land \forall n >n_0: x|_n \in \bigcup S(n)\}).\]
Since $|S(n)|\leq g(n)$, we have $|\bigcup S(n)|\leq g(n)(g(n)+1)\leq f(n)$, hence $A$ has levels of size $\leq f+1$. It remains to check that $A$ has nonempty intersection with each $X \in \mathcal{F}$.

Let $n_0 \in \omega$ be such that $\forall n > n_0: y^{X}(n)\in S(n)$. Since $y^{X}_{g(n_0)}|_{n_0}=\eta|_{n_0}$ and $\forall n > n_0: y^{X}_{g(n_0)}|_n \in y^{X}(n) \subseteq \bigcup S(n)$, we see that $y^{X}_{g(n_0)} \in A \cap X$.
\end{proof}

\pagebreak

Considering upper bounds, we begin by observing that $\mathfrak{ge}(\text{Laver})$ is trivially bounded from above by the generic existence number of measure zero ultrafilters, as each $\mathcal{I}_f$ consists of sets with closure of measure zero. Borrowing from Brendle and Fla\v{s}kov\'a \cite{brefla17}, we denote by $\textsf{mz}$ the ideal of sets with closure of measure zero, and by $\mathcal{E}$ the $\sigma$-ideal generated by $\textsf{mz}$. Furthermore, as is standard, $\mathfrak{d}=\min\{|\mathcal{F}|: \forall g \in \omega^\omega \:\exists f \in \mathcal{F}: g \leq^{\ast}f\}$. Brendle \cite{brendle-between} proved the following

\begin{fct}[\cite{brendle-between}, Theorem D]
$\mathfrak{ge}(\textsf{mz})=\max\{\text{non}(\mathcal{E}), \mathfrak{d}\}.$
\end{fct}

It follows that

\begin{fct}\label{fct:upperbound}
$\mathfrak{ge}(\text{Laver})\leq \max\{\text{non}(\mathcal{E}), \mathfrak{d}\}$.
\end{fct}

Recall that $\mathcal{SN}$ denotes the ideal of strong measure zero subsets of $2^{\omega}$, and $\text{non}(\mathcal{SN})$ is the smallest cardinality of a set which is not strong measure zero. Miller \cite{Miller81} characterised $\text{non}(\mathcal{SN})$ as the smallest cardinality of a bounded family $\mathcal{D}\subseteq\omega^\omega$ such that for all $g\in\omega^\omega$, there is $f\in\mathcal{D}$ for which the set $\{n\in\omega : g(n)=f(n)\}$ is finite. Using a similar argument to Lemma 10 of Canjar \cite{Canjar-gen} and Proposition 3.18 of Brendle and Fla\v{s}kov\'a \cite{brefla17}, we will obtain $\mathfrak{ge}(\text{Laver})\leq \text{non}(\mathcal{SN})$. In the following, we identify $\mathbb{Q}$ with the set $2^{<\omega}$.

\begin{prp}\label{prop:upperboundnon}
$\mathfrak{ge}(\text{Laver})\leq\text{non}(\mathcal{SN})$.
\end{prp}

\begin{proof}
Let $\kappa<\mathfrak{ge}(\text{Laver})$ be a cardinal and $h\in\omega^\omega$. Let $\mathcal{D}\subseteq \prod h$ be a family of size $\kappa$. We will find some $u\in\prod h$ such that the set $\{n\in \omega: h(n)=f(n)\}$ is infinite for every $f\in \mathcal{D}$, which shows that $\kappa < \text{non}(\mathcal{SN})$.

Partition $\omega$ into intervals $I_n$, each of size $n^2$, and partition each $I_n$ into intervals $J^n_i$, $i<n$, each of size $n$. Fix some strictly increasing sequence $\langle k_n: n \in \omega \rangle$ such that $\forall n \in \omega: 2^{k_n}\geq |\prod_{i\in I_n} h(i)|$, and let $\varphi_n : \prod_{i\in I_n} h(i) \to 2^{k_n}$ be an injection.

For each $f\in \mathcal{D}$, define
\[A_f=\{\varphi_n(t) \in 2^{k_n} : n\in\omega \land \; t\in\dom(\varphi_n) \land \; \exists j<n \; (t|_{J^n_j}=f|_{J^n_j})\}.\]
Observe that the family $\mathcal{F}=\{A_f : f\in \mathcal{D}\}$ is a filter base of cardinality $\kappa$. Find a non-decreasing unbounded function $g\in\omega^\omega$ such that $g(k_0)=0$ and $g(k_n)=n-1$ for all $n>0$. Since $\kappa<\mathfrak{ge}(\text{Laver})$, we find $A\in [\mathbb{Q}]^\omega$, with levels of size $\leq g+1$, such that $\mathcal{F}\cup\{A\}$ still generates a filter.

For each $n\in\omega$, list $A\cap 2^{k_n}$ as $\{s^{k_n}_i: i < n\}$ (allowing repetitions). Define $u \in\prod h$ such that $u$ agrees with $\varphi^{-1}_n(s^{k_n}_i)$ on the $i$'th element of $J^n_j$ for all $j<n$ (define $\varphi^{-1}_n(s^{k_n}_i)$ arbitrarily if $s^{k_n}_i\notin \ran(\varphi_n)$). We claim that for every $f\in\mathcal{D}$, there are infinitely many $k$ such that $u(k)=f(k)$.

Indeed, let $f\in\mathcal{D}$ and find $\varphi_n(t)\in A_f\cap A$ for some $n\in\omega$ and $t\in\prod_{i\in I_n} h(i)$. Say $\varphi_n(t)$ is the $i$'th element in our list of $A\cap 2^{k_n}$. Then, by definition, $u$ agrees with $t$ on the $i$'th member of $J^n_j$ for every $j<n$. Since $t|_{J^n_j}=f|_{J^n_j}$ for some $j<n$, it follows that $u$ agrees with $f$ on some point in the interval $I_n$. As $A_f\cap A$ is infinite, we are done.
\end{proof}

Observe that a minor rephrasing of the preceding proof actually shows that the generic existence of \emph{rapid} ultrafilters implies $\text{non}(\mathcal{SN})=\mathfrak{c}$. This is because we only require the above set $A$ to intersect each of the finite sets $\prod_{i \in I_n}h(i)$ in at most $n$ points. Such an $A$ can clearly be obtained if the filter base $\mathcal{F}$ is extendable to a rapid ultrafilter.

Combining Propositions \ref{prp:lowerboundcov}, \ref{prp:lowerboundnon} and \ref{prop:upperboundnon}, as well as Fact \ref{fct:upperbound}, we have

\begin{cor}\label{cor:inequalities}
\[\text{cov}(\mathcal{M}),\text{non}(\mathcal{NA})\leq\mathfrak{ge}(\text{Laver})\leq\text{non}(\mathcal{SN}),\max\{\text{non}(\mathcal{E}), \mathfrak{d}\}.\]
\end{cor}

As there are no rapid filters in the Mathias-, Laver- and Miller models, Laver ultrafilters do not exist in them as well. Shortly, we will see that they do not exist in the Silver model either. For three of the remaining classical models, we have the following:
\begin{cor}
Laver ultrafilters generically exist in the Cohen model, and do not generically exist in the random- and Sacks models\footnote{Note, however, that they do \emph{exist} in the Sacks model, as Ramsey ultrafilters exist there.}.
\end{cor}
\begin{proof}
As $\text{cov}(\mathcal{M})=\mathfrak{c}$ in the Cohen model, and $\text{non}(\mathcal{SN})\leq\text{non}(\mathcal{N})=\omega_1$ in the random- and Sacks models.
\end{proof}

\begin{rmk*}
Before proceeding to our next results, let us remark that none of the inequalities in Corollary \ref{cor:inequalities} are tight:
\begin{enumerate}[label=(\roman*)]
    \item  Cardona, Mej\'ia and Rivera-Madrid \cite[Theorem D and Theorem E)]{CARDONA_MEJÍA_RIVERA-MADRID_2025} constructed models in which $\text{non}(\mathcal{NA})<\text{cov}(\mathcal{M})$ holds, hence $\text{non}(\mathcal{NA})\leq\mathfrak{ge}(\text{Laver})$ is consistent.
    \item In our model of (no $P$-points) $+$ (generic existence of Laver ultrafilters), given in Theorem \ref{thm:lavernop-point} below, we have $\text{cov}(\mathcal{M})<\text{non}(\mathcal{NA})=\mathfrak{ge}(\text{Laver})=\mathfrak{c}$.\footnote{Of course, simply separating these cardinal characteristics does not necessitate the non-existence of $P$-points, since the partial order used in Theorem \ref{thm:lavernop-point} to force $\text{non}(\mathcal{NA})=\mathfrak{c}$ is $\omega^\omega$-bounding.}
    \item In the remainder of this section, we will see that there are no Laver ultrafilters in the Silver model. The same argument can be adapted to show that there are no Laver ultrafilters in the model obtained by a length $\omega_2$ iteration of the \emph{bounded-below-$2^n$ Silver forcing} over a model of \textsf{CH}. However, bounded Silver forcing makes the ground reals strong measure zero; therefore, $\mathfrak{ge}(\text{Laver})<\text{non}(\mathcal{SN})=\mathfrak{c}$ holds in this model (see, e.g., 7.4.C in \cite{bj95} for the definition of this forcing and a proof of this fact).
    \item Since there are no Laver ultrafilters in the Mathias model, it satisfies $\mathfrak{ge}(\text{Laver})<\text{non}(\mathcal{E})=\mathfrak{d}=\mathfrak{c}$ (see, e.g., Lemma 7.4.2 in \cite{bj95}).
\end{enumerate}
\end{rmk*}

It turns out that a very similar argument to the one of Chodounsk\'y and Guzm\'an \cite{chodguz-ppoints} shows that Laver ultrafilters do not exist in the Silver model. We will sketch the argument below, and the interested reader can find the full proof in \cite{chodguz-ppoints}.

Let us denote Silver forcing by $\mathbb{SI}$, which consists of partial functions $p:\dom(p)\to 2$ with $\dom(p)\subseteq \omega$ co-countable, ordered by containment. Recall that Silver forcing is proper and has the \emph{Sacks property}, where a forcing notion is said to have the Sacks property if it has the Laver property and does not add unbounded reals. In the following, we will borrow the notation of \cite{chodguz-ppoints}.

\begin{thm}\label{silverthm}
Assume that $\mathbf{V}\vDash \textsf{CH}$ and let $\mathbb{P}_{\omega_2}=\langle \mathbb{P}_{\alpha}, \undertilde{\mathbb{Q}}_{\beta} : \alpha\leq\omega_2, \beta<\omega_2\rangle$ be the countable support iteration of Silver forcing, where $\mathbb{P}_{\alpha}\Vdash \undertilde{\mathbb{Q}}_{\alpha}=\undertilde{\mathbb{SI}}$ for all $\alpha<\omega_2$. Let $G\subseteq \mathbb{P}_{\omega_2}$ be generic. Then there are no Laver ultrafilters in $\mathbf{V}[G]$.
\end{thm}

By well-known theorems of Shelah (see, e.g., Chapter VI of \cite{she}), any countable support iteration of Silver forcing also has the Sacks property. Also, by classical arguments, for any ultrafilter $\mathcal{U}\in \mathbf{V}[G]$, there are stationarily many $\alpha<\omega_2$ with uncountable cofinality such that $\mathcal{U}\cap \mathbf{V}[G_{\alpha}]\in \mathbf{V}[G_{\alpha}]$ is an ultrafilter, where $G_{\alpha}$ is the corresponding generic for $\mathbb{P}_{\alpha}$. Therefore, to prove Theorem \ref{silverthm}, it will suffice to prove the upcoming proposition, but first we recall some definitions from \cite{chodguz-ppoints}.
\begin{defi}
\begin{enumerate}[label=(\roman*)]
\item For a possibly partial function $p:\omega \to 2$ and $n\in\omega$, define $I_n(p)=\{k\in\omega : |k\cap p^{-1}(1)|=n\}$ and $\mathcal{I}(p)=\{I_n(p) : n\in\omega\}$, the corresponding interval partition. For such (partial) functions $p$, $\dom(p)$ denotes the domain of $p$ and $\operatorname{cod}(p)$ denotes the complement of $\dom(p)$ in $\omega$.
\item For positive integers $n$, $\equiv_n$ denotes congruence modulo $n$ and $-_n$ denotes subtraction modulo $n$. For $m\in\omega$ and $A\subseteq \omega$, $m\in_n A$ means $A$ contains a natural number congruent to $m$ modulo $n$.
\item For $n\in\omega$, $k(n)\in\omega$ denotes the least natural number $k$ such that for every $C\in[k]^{n^2}$, there is $s\in k$ such that $C\cap (C-_{k} \{s\})=\varnothing$.
\item Define $v,m\in\omega^\omega$ as follows. $v(0)=0$ and $m(0)=k(2)$. Assuming that $v(n-1)$ and $m(n-1)$ are defined, put $v(n)=\sum_{i<n} m(i)$ and $m(n)=k((n+1)(v(n)+2))$.
\end{enumerate}
\end{defi}

\begin{prp}\label{prp:silver}
Suppose $\mathcal{U}$ is an ultrafilter, and $\undertilde{\mathbb{R}}$ is an $\mathbb{SI}$-name for a forcing poset with the Sacks property. Then $\mathcal{U}$ cannot be extended to a Laver ultrafilter in any $(\mathbb{SI}*\undertilde{\mathbb{R}})$-generic extension.
\end{prp}
\begin{proof}

We shall use our first characterisation of Laver ultrafilters from Definition \ref{def:laver}. Working in the generic extension, let $s:\omega \to 2$ denote the Silver generic real and for $n\in\omega$ and $i<m(n)$, define
$$\mathcal{D}^n_i=\bigcup\{I_j(s) : j\equiv_{m(n)}i\},$$
and $\mathcal{D}^n=\{\mathcal{D}^n_i : i<m(n)\}$. We claim that the sequence of partitions $\{\mathcal{D}^{n} : n\in\omega\}$ witnesses the failure of Laverness of any ultrafilter extending $\mathcal{U}$ in the generic extension by $\mathbb{SI}*\undertilde{\mathbb{R}}$. It suffices to prove the following.
\begin{claim}
Assume that there is an $(\mathbb{SI}*\undertilde{\mathbb{R}})$-name $\undertilde{Z}$ for a subset of $\omega$, a name $\undertilde{h}$ for a function such that $\undertilde{h}(n)\in[m(n)]^n$ for all $n\in\omega$, and a condition $\langle p,\undertilde{q}\rangle\in \mathbb{PS}*\undertilde{\mathbb{R}}$ such that $$\langle p,\undertilde{q}\rangle \Vdash (\forall n)(\forall i< m(n)) (\undertilde{Z}\cap \mathcal{D}^n_i\neq \varnothing \Rightarrow i\in \undertilde{h}(n)).$$
Then there is some $U\in\mathcal{U}$ and $\langle p', \undertilde{q}\rangle\leq \langle p, \undertilde{q}\rangle$ such that $\langle p', \undertilde{q}\rangle \Vdash U\cap \undertilde{Z}=\varnothing$.
\end{claim}
\begin{proof}[Proof of claim]
Assume $\undertilde{Z}, \undertilde{h}, \langle p,\undertilde{q}\rangle$ are as in the statement. By the Sacks property, we may assume that there is a function $P$ defined over $\omega$ in the ground model such that $\forall n \in \omega: \; P(n)\in [m(n)]^{n^2}$, satisfying
$$\langle p,\undertilde{q}\rangle \Vdash (\forall n)(\forall i< m(n)) (\undertilde{Z}\cap \mathcal{D}^n_i\neq \varnothing \Rightarrow i\in P(n)).$$
Choose an interval partition $\mathcal{J}=\{J_i : i\in\omega\}$ such that for all $n\in\omega$ and $j\in 2$, we have $|J_{2n+j}\cap \operatorname{cod}(p)|> m(n)$. Assume without loss of generality that $W=\bigcup_{n\in\omega}J_{2n+1}\in\mathcal{U}$.

Let $p_1\leq p$ be an extension such that $J_{2n+1}\subseteq \dom(p_1)$ for all $n\in\omega$, and $|\operatorname{cod}(p)\cap J_{2n}|=m(n)$ for all $n$. Note that $|\operatorname{cod}(p_1)\cap \min(J_{2n})|=v(n)$ for every $n\in\omega$. Putting $C_n=P(n)-_{m(n)}(v(n)+2)$, we have $|C_n|\leq n^2\cdot(v(n)+2)$ for each $n\in\omega$. Finally, let $H_n=J_{2n+1}\cap \bigcup\{I_j(p_1) : j\in_{m(n)} C_n\}$ for all $n\in\omega$. As in \cite{chodguz-ppoints}, we will split the proof into two cases now. We only present the first case here, the reader may read the proof of Proposition 5 in \cite{chodguz-ppoints} for the argument of the other case.

\emph{Case 1.} Assume that $U=\bigcup \{J_{2n+1}\setminus H_n : n\in\omega\}\in\mathcal{U}$. Find $p_2\leq p_1$ such that
\begin{enumerate}[label=(\roman*)]
    \item $p_2^{-1}(1)=p_1^{-1}(1)$,
    \item $|\operatorname{cod}(p_2)\cap J_{2n}|=1$ for all $n\in\omega$.
\end{enumerate}
But then,
\begin{enumerate}[label=(\roman*)]
    \item $\mathcal{I}(p_2)=\mathcal{I}(p_1)$,
    \item $|\operatorname{cod}(p_2)\cap \min(J_{2n+1})|=n+1$,
    \item If $I_j(p_2)\subseteq J_{2n+1}$, then $I_j(p_2)\subseteq \operatorname{dom}(p_2)$.
\end{enumerate}
It follows that, in the generic extension by $\mathbb{SI}*\undertilde{\mathbb{R}}$, for the unique $i\leq n+1$ with $I_j(p_2)=I_{j+i}(s)$, if $\undertilde{Z}\cap I_{j+i}(s)\neq\varnothing$, then $j\in_{m(n)} C_n$. Therefore, $\langle p_2, \undertilde{q}\rangle\Vdash \undertilde{Z}\subseteq (\bigcup_{n\in\omega} H_n)\cup(\bigcup_{n\in\omega}J_{2n})$. This concludes the case since $U\cap ((\bigcup_{n\in\omega} H_n)\cup(\bigcup_{n\in\omega}J_{2n}))=\varnothing$.

\emph{Case 2.} This is the case of $\bigcup_n H_n\in\mathcal{U}$, and we refer the reader to the corresponding case in \cite{chodguz-ppoints}.
\end{proof}
This concludes the proof of Proposition \ref{prp:silver}.
\end{proof}

Note that the Silver model contains $Q$-points, and hence rapid ultrafilters. It follows that the existence of a rapid ultrafilter does not imply the existence of a Laver ultrafilter.

Finally, we prove that it is consistent that there are no $P$-points while Laver ultrafilters exist generically. By Proposition \ref{prp:lowerboundnon}, it suffices to find a model without $P$-points in which $\text{non}(\mathcal{NA})=\mathfrak{c}$ holds.

\begin{thm}\label{thm:lavernop-point}
It is consistent that there are no $P$-points while Laver ultrafilters exist generically.
\end{thm}

To prove this, we will interleave two rather well-known forcing notions in a countable support iteration of length $\omega_2$, both originally introduced by Shelah. Therefore, before we start the proof of Theorem \ref{thm:lavernop-point}, we first recall these two forcing notions in the following two lemmas:

\begin{lem}[Shelah \cite{she}, Chapter VI]
Let $\mathcal{U}$ be a $P$-point. There is a proper forcing notion $\mathbb{Q}(\mathcal{U})$ of cardinality $\mathfrak{c}$ with the following properties:
\begin{enumerate}[label=(\roman*)]
    \item $\mathbb{Q}(\mathcal{U})$ is $\omega^\omega$-bounding (i.e., it does not add unbounded reals), and
    \item For any $\mathbb{Q}(\mathcal{U})$-name $\undertilde{\mathbb{R}}$ for a forcing notion such that $\mathbb{Q}(\mathcal{U})\Vdash ``\undertilde{\mathbb{R}} \;\text{is $\omega^\omega$-bounding}"$, we have $\mathbb{Q}(\mathcal{U})*\undertilde{\mathbb{R}}\Vdash ``\mathcal{U} \; \text{cannot be extended to a $P$-point}"$.
\end{enumerate}
\end{lem}
\begin{lem}[Shelah \cite{she-vive}, also see \cite{abraham2009proper} and \cite{laflamme-zap} for slightly different representations]
Let $b\in \omega^\omega$ and set $h(n):=n+1$. There is a proper, $\omega^\omega$-bounding forcing notion $\mathbb{Q}_b$ of cardinality $\mathfrak{c}$ such that
\begin{equation}\label{eq:genericslalom}
\mathbb{Q}_b\Vdash \exists S \in \mathcal{S}(b,h)\; \forall x \in \left(\prod b\right)^{\mathbf{V}}: x \in^{\ast}S. \tag{\textasteriskcentered}
\end{equation}
\end{lem}
\begin{proof}
Let $b\in\omega^\omega$. We may assume that $b$ is strictly increasing. Set $\mathbb{T}_b=\bigcup_{n\in\omega}\prod_{k<n} [b(k)]^{k+1}$. For $k,m\in\omega$, call a set $A\subseteq [b(k)]^{k+1}$ \emph{$m$-big} if for every $x\subseteq b(k)$ of size $\leq m$, there is $a\in A$ such that $x\subseteq a$. A perfect subtree $T\subseteq \mathbb{T}_b$ will be called an \emph{$m$-tree} if $\operatorname{succ}_T(t)$ is $m$-big for every $\operatorname{stem}(T)\subseteq t\in  T$.

We define the partial order $\mathbb{Q}_b$ to consist of perfect subtrees $T\subseteq \mathbb{T}_b$ such that for any $m\in\omega$, there is $n\in\omega$ that satisfies: $t\in T \land |t|\geq n\implies \operatorname{succ}_T(t)$ is $m$-big. We order $\mathbb{Q}_b$ via inclusion. Using the following observation, it is not hard to check that $\mathbb{Q}_b$ is proper and $\omega^\omega$-bounding (see, e.g., Lemma 3.12 of \cite{abraham2009proper}):
\begin{obs}
If $m>0$, $n\in b(k)$, and $A\subseteq [b(k)]^{k+1}$ is $(m+1)$-big, then $\{a\in A : n\in a\}$ is $m$-big.
\end{obs}
To finish the proof, we show that the generic slalom $\undertilde{S}\in\mathcal{S}(b,h)$ introduced by $\mathbb{Q}_b$, i.e., the generic branch through $\mathbb{T}_b$, satisfies (\ref{eq:genericslalom}).

For this purpose, let $x\in \prod b$ be a real in the ground model and let $T\in\mathbb{Q}_b$ be a condition. Set $s_0=\operatorname{stem}(T)$ and $|s_0|=n_0$. By extending if necessary, we may assume that $T$ is a $2$-tree. Using the previous observation, we may inductively define $T'\leq T$ such that
\begin{enumerate}[label=(\roman*)]
    \item $T'\in \mathbb{Q}_b$,
    \item $\operatorname{stem}(T')=s_0$ and $T'$ is a $1$-tree,
    \item For all $s_0\subseteq s\in T'$ and $n_0\leq n\in\operatorname{dom}(s): x(n)\in s(n)$.
\end{enumerate}
It follows that $T'\Vdash \forall n\geq n_0 :\; x(n)\in \undertilde{S}(n)$.
\end{proof}
Using these facts, we now conclude the proof of Theorem \ref{thm:lavernop-point}.

\begin{proof}[Proof of Theorem \ref{thm:lavernop-point}]
We start with a ground model satisfying $\mathbf{V}\vDash \mathfrak{c}=\omega_1 \land 2^{\omega_1}=\omega_2 \land \diamondsuit_{\{\alpha<\omega_2\; :\; \operatorname{cf}(\alpha)=\omega_1\}}$. Using standard bookkeeping techniques, we define $\langle \mathbb{P}_{\alpha}, \undertilde{\mathbb{Q}}_{\beta} : \alpha\leq\omega_2, \beta<\omega_2\rangle$, a countable support iteration of proper and $\omega^\omega$-bounding forcing notions of length $\omega_2$, as follows. At stages $\alpha<\omega_2$ of uncountable cofinality, we let $\undertilde{\mathcal{U}}_{\alpha}$ be a $\mathbb{P}_{\alpha}$-name for a $P$-point coded by the $\diamondsuit_{\{\alpha<\omega_2\; :\; \operatorname{cf}(\alpha)=\omega_1\}}$ sequence, if the diamond sequence indeed codes such a name at $\alpha$. If not, we let $\undertilde{\mathcal{U}}_{\alpha}$ be a name for any $P$-point (this is possible because \textsf{CH} holds in the extension by $\mathbb{P}_{\alpha}$). We then define $\undertilde{\mathbb{Q}}_{\alpha}$ to be a $\mathbb{P}_{\alpha}$-name for $\mathbb{Q}(\undertilde{\mathcal{U}_{\alpha}})$.

For the other stages, again by standard bookkeeping techniques, we find in the ground model an enumeration $\langle \undertilde{b}_{\alpha} : \alpha\in\{\alpha<\omega_2 : \operatorname{cf}(\alpha)\neq\omega_1\}\rangle$ such that each $\undertilde{b}_{\alpha}$ is a $\mathbb{P}_{\alpha}$-name for a function in $\omega^\omega$, and every such $\mathbb{P}_{\omega_2}$-name appears in the sequence cofinally often. Then, at stage $\alpha\in\{\alpha<\omega_2 : \operatorname{cf}(\alpha)\neq\omega_1\}$, we define $\undertilde{\mathbb{Q}}_{\alpha}$ to be a $\mathbb{P}_{\alpha}$-name for $\mathbb{Q}_{\undertilde{b}_{\alpha}}$. Therefore, by the previous two lemmas and Lemma \ref{lem-non(na)char}, there are no $P$-points in the generic extension by $\mathbb{P}_{\omega_2}$, but $\text{non}(\mathcal{NA})=\mathfrak{c}=\omega_2$ holds. Consequently, by Proposition \ref{prp:lowerboundnon}, Laver ultrafilters generically exist in the extension.
\end{proof}

\section*{Open Problems}

We conclude the paper by listing some problems that remain unsolved.

\begin{problem}
Does \textsf{MA} imply the existence of a hereditarily rapid, countable closed ultrafilter that is not a Laver ultrafilter?
\end{problem}

\begin{problem}
Is $\max\{\text{non}(\mathcal{NA}), \text{cov}(\mathcal{M})\}<\mathfrak{ge}(\text{Laver})$ consistent?
\end{problem}

As the reader may have observed -- although we showed that Laver ultrafilters do not exist generically in the random model -- we were unable to determine whether they exist in the random model at all.

\begin{problem}
Do Laver ultrafilters exist in the random model?
\end{problem}


\let\OLDthebibliography\thebibliography
\renewcommand\thebibliography[1]{
  \OLDthebibliography{#1}
  \setlength{\parskip}{0pt}
  \setlength{\itemsep}{3pt plus 0.3ex}
}

\bibliographystyle{plain}
\bibliography{rapid}

\end{document}